\documentclass[10pt]{amsart}
\usepackage[latin1]{inputenc}
\usepackage{amsmath}
\usepackage{amsthm}
\usepackage{amssymb} 
\usepackage{mathtools}
\usepackage{amsfonts, verbatim,tikz,bbold, tikz-cd}
\usepackage[all,cmtip]{xy}
\usepackage[dvipsnames]{xcolor}  
\usepackage{multirow}
\usepackage[
    pagebackref=true,
    colorlinks=true,
    citecolor=Blue,     
    linkcolor=Maroon,   
    urlcolor=teal
]{hyperref}
\usepackage{microtype, wrapfig,textcomp,mathrsfs,csquotes,fbb}
\newtheorem{theorem}{Theorem}[section]
\newtheorem{corollary}[theorem]{Corollary}
\newtheorem{lemma}[theorem]{Lemma}  
\newtheorem{remark}[theorem]{Remark}
\newtheorem{prop}[theorem]{Proposition} 
\newtheorem{ques}[theorem]{Question}

\newtheorem{example}[theorem]{Example}

\usepackage{cleveref}
\crefalias{corollary}{corollary}
\crefalias{theorem}{theorem}
\crefalias{lemma}{lemma}
\crefalias{prop}{prop}
\crefalias{ques}{ques}
\crefalias{defn}{defn}

\newcommand{\C}{\mathbb C}

\newcommand{\nc}{\newcommand}

\nc{\thm}{\theorem}
\nc{\cor}{\corollary}
\nc{\mc}{\mathcal}
\nc{\mb}{\mathbb}
\nc{\mf}{\mathfrak}
\nc{\ul}{\underline}
\nc{\ol}{\overline}
\nc{\N}{\mb N}
\nc{\R}{\mb R}
\nc{\rP}{\mb P}
\nc{\Z}{\mb Z}
\nc{\Q}{\mb Q}
\nc{\F}{\mb F}
\nc{\gm}{\gamma}
\nc{\Hom}{Hom}
\nc{\SO}{SO}
\nc{\Spin}{Spin}
\nc{\ord}{\mathrm{ord}}
\nc{\X}{\mf{X}}
\nc{\lip}{\langle}
\nc{\rip}{\rangle}
\nc{\dm}{{D}_{2m}}
\nc{\dn}{{D}_{2n}}
\nc{\gl}{\mathrm{GL}_2(\mathbb{C})}
\nc{\mbz}{\mathbb{Z}_2}

\nc{\dmo}{\DeclareMathOperator}
\dmo{\GL}{GL}
\dmo{\sgn}{sgn}
\dmo{\res}{res}
\dmo{\diag}{diag}
\dmo{\Or}{O}
\dmo{\Pin}{Pin}
\title[Chern Classes of Representations of Dihedral and Quaternion groups]{Realizable Cohomology Classes of Dihedral and Generalized Quaternion Groups} 
\author{Sutirtha Datta, Jyotirmoy Ganguly, and Rohit Joshi}
\date{}

\address{IISER Pune, Dr. Homi Bhabha Road,Pune-411008, Maharashtra, India}
\email{sutirtha2702@gmail.com}

\address{MURTI Research Center, GITAM (Deemed to be University),\\
	NH 207, Nagadenahalli, Doddaballapura Taluk, Bengaluru Rural,\\
	Karnataka - 562163, India.}
\email{jyotirmoy.math@gmail.com}  

\address{IISER Pune, Dr. Homi Bhabha Road,Pune-411008, Maharashtra, India} \email{rohitsj2004@gmail.com} 
\begin{document}
 
\begin{abstract}
    We compute the Chern classes of irreducible representations for the dihedral groups $D_{2n}$ and the generalized quaternion groups $Q_{2^n}$. Utilizing these computations, we investigate whether every even-dimensional class in the integral group cohomology of these families can be realized as a Chern class of some complex representation. It turns out that this realizability question answers in the affirmative for the generalized quaternion groups. In contrast, it turns out that the answer for dihedral groups is erratic.
\end{abstract}
 \maketitle

 \tableofcontents

\section{Introduction}

For a finite group $G$, consider a complex representation $\rho:G\to \GL_n(\C)$. The integral cohomology ring of the classifying space $BGL_n(\C)$ is a polynomial ring $\mathbb{Z}[c_1, c_2, \dots, c_n]$, where $c_k\in H^{2k}(BGL_n(\C),\Z)$. The reprsentation $\rho$ induces a  map on classifying spaces $\rho^*: H^*(BGL_n(\C), \mathbb{Z}) \to H^*(G, \mathbb{Z})$ that allows us to define the $i$-th \emph{Chern class} of $\rho$ as 
$$c_i(\rho) := \rho^*(c_i) \in H^{2i}(G, \mathbb{Z}) \quad ; \quad 1 \leq i \leq  n.$$
For exposition on Chern classes of representations of connected reductive complex algebraic groups in terms of highest weights, see \cite{joshi2026characteristic}.

A natural question is which cohomology classes arise as Chern classes of representations of $G$. Accordingly, we introduce the following definition. 
For a finite group $G$, we take
$$\mathcal{RC}(G)=\{\alpha \in \mathcal{H}^{\text{even}}(G, \mb Z) \mid  \alpha = c_k(\pi)\, \text{for some representation $\pi$ of $G$ and $k>0$}\},$$
where $\mathcal{H}^{\text{even}}(G,\Z)=\bigcup_{k> 0}H^{2k}(G,\Z).$
We call the elements of $\mathcal{RC}(G)$ to be \textit{realizable} cohomology classes for the group $G$. We write $H_{C}^*(G,\Z)$ to denote the subring of $H^*(G, \Z)$ generated by the elements of $\mathcal{RC}(G)$. 
With the notations introduced above we ask the following question.

\begin{ques}\label{ques1}
    For a finite group $G$, is $\mathcal{RC}(G)=\mathcal{H}^{\text{even}}(G,\Z)$?
    In other words, is every even-degree integral cohomology class $\alpha \in H^{2k}(G, \mathbb{Z})$ realizable ?
\end{ques}
For expositions in the similar lines, see \cite{naolekar2012realizing}.
For orthogonal representations of a finite group $G$,  we define the following in an analogous manner:
$$\mathcal{RW}(G)=\{ w_i(\pi)\in \mathcal{H}(G,\Z_2)\mid \pi\,\, \text{is an orthogonal representation of $G$},\,\,i\geq 1\},$$
where $\mathcal{H}(G,\Z_2)=\bigcup_{k\geq 1}H^k(G,\Z_2).$ We call the elements of $\mathcal{RW}(G)$ to be \textit{realizable} cohomology classes with $\Z_2$ coefficients for the group $G$. One may ask a similar question to Question \ref{ques1} for Stiefel Whitney classes of orthogonal representations of a group $G$. It is worth mentioning that the sets $\mathcal{RC}(G)$ and $\mathcal{RW}(G)$ admit no group structure.

In this paper, we explicitly compute the Chern classes of irreducible representations for both the dihedral groups $D_{2n}$ and the generalized quaternion groups $Q_{2^n}$. For $Q_{2^n}$ the computation yields a particularly clean and explicit answer (see \Cref{thm: total CC of Q_2^n}). The dihedral groups, on the contrary, require a case by case analysis (see \Cref{sec: chern classes of d2n}). 

As an application of these computations, we settle the realizability problem \ref{ques1} for these families, providing an affirmative answer for the generalized quaternion groups. 
\begin{theorem}\label{thm: every cohomology class of Q_2^n is realizable}
    For $Q_{2^n}$, we realize every cohomology class as the Chern class for some representation.
\end{theorem}
For, $D_{2n}$, there is no such general result. The relevant results on Dihedral groups are established in \Cref{thm: D10 realizable}, Corollary \ref{cor: D_6m, m odd} and Corollary \ref{thm: non realizable class for D8k, k odd}.

\textbf{Layout of the article}: 
In Section \ref{sec: Notation and prelim}, we present an overview of the representation theory of the dihedral groups and generalized quaternion groups. In \Cref{sec: klein 4 group} we inspect the realizable cohomology classes of the Klein $4$-group. The chern classes for the irreducible representations of the Dihedral groups are computed in Section \ref{sec: chern classes of d2n} and we answer the realizability question in Section \ref{sec: realizability of chern classes of D2n}. Section \ref{sec: Chern classes for Q2n} is devoted to figuring out the Chern classes of irreducible representations of $Q_{2^n}$ and the key result is \Cref{thm: total CC of Q_2^n}. In Section \ref{sec: realizability of Q2n}, the realizability question for $Q_{2^n}$ is addressed and the final result \Cref{thm: every cohomology class of Q_2^n is realizable} states that for $Q_{2^n}$, we realize every cohomology class as the Chern class for some representation. Finally, in \Cref{sec: realizable SW classes}, we explore the mod $2$ cohomology classes which are realizable as Stiefel-Whitney Classes.

{\bf Acknowledgements:} The authors would like to thank Steven Spallone for helpful conversations.

\section{Notations and Preliminaries}\label{sec: Notation and prelim}

\subsection{Dihedral Groups}~~\vspace{1ex}

Let $\dn$ be the dihedral group of order $2n$ with the presentation $$\dn= \langle r,s \mid r^n=s^2=1, rsrs=1\rangle.$$

For $n$ even, the group $D_{2n}$ has four linear characters
$$\mathbb{1} : (r,s) \mapsto (1,1),$$
$$\chi_r : (r,s) \mapsto (-1,1),$$
$$\chi_s: (r,s) \mapsto (1,-1),$$
$$\chi_{rs}: (r,s) \mapsto (-1,-1).$$
We denote the irreducible degree $2$ representations of $D_{2n}$ by $\sigma_k$, where $1\leq k\leq n/2-1$. They are defined as
$$\sigma_k(r) = \begin{pmatrix}
  \zeta_n^k & 0      \\
    0 & \zeta_n^{-k}
\end{pmatrix} , \sigma_k (s)= \begin{pmatrix}
    0 &1 \\
    1 &0
\end{pmatrix},$$
where $\zeta_n$ denotes the primitive $n$-th root of unity.

For $n$ odd, there are only $2$ linear characters $\mathbb{1}$ and $\chi_s$. In this case the $2$ dimensional irreducible representations are given by $\sigma_k$ as defined above for $1\leq k\leq (n-1)/2$.

\begin{prop}
For $n$ even, we have
\begin{enumerate}
\item
$\sigma_k\otimes \chi_r\cong \sigma_{n/2-k}$.
\item
$\sigma_k\otimes \chi_s\cong \sigma_k$.
\item
$\sigma_k\otimes \chi_{rs}\cong \sigma_{n/2-k}$.
\end{enumerate}
\end{prop}

\begin{proof}
   Follows from \cite[Theorem 4.4.3]{hirzebruch1966topological}.
\end{proof}

  \vspace{1ex}
\subsection{Generalized Quaternion Groups} ~~\vspace{1ex}

Here we review the representation theory for the generalized quaternion group $Q_{2^n}$ of order $2^n$ with presentation:
$$Q_{2^n}=\langle x,y\mid x^{2^{n-1}}=y^4=1, x^{2^{n-2}}=y^2, y^{-1}xy=x^{-1} \rangle,$$
for $n\geq 2$. See \cite[Section 1]{pitt1973free} for reference.

The group $Q_{2^n}$ has four linear characters:
$$\mathbb{1}: (x,y)\rightarrow (1,1),$$
$$\epsilon: (x,y)\rightarrow (1,-1),$$
$$\eta: (x,y)\rightarrow (-1,1),$$
$$\epsilon\eta: (x,y)\rightarrow (-1,-1).$$

\begin{comment}
Another way to think about this action is by considering the $2$-dimensional representation $\sigma_1 : Q_{2^n} \to U(2).$ Since the image of $\sigma_1$ is $SU(2)$, we get a free action of $Q_{2^n}$ on $S^3.$
The degree of irreducible representations of $Q_{2^n}$ is atmost $2$. One can easily check that $A=\langle x \rangle$ is an index $2$ abelian subgroup of $Q_{2^n}$. 
\end{comment}
The $2$-dimensional irreducible representations $\rho_k$ of $Q_{2^n}$ are enumerated by $k$ where $1\leq k\leq 2^{n-2}-1$. They are defined as
$$ \rho_k(x)= \begin{pmatrix}
    \zeta_{2^{n-1}}^k & 0\\0 &\zeta_{2^{n-1}}^{-k} 
\end{pmatrix}, \quad \quad \rho_k(y)= \begin{pmatrix}
    0 & (-1)^k\\
    1 & 0
\end{pmatrix}$$
 where $\zeta_{2^{n-1}}$ is the primitive $2^{n-1}$-st roots of unity.
 
\begin{prop}
We have
\begin{enumerate}
\item
$\rho_k\otimes \epsilon\cong \rho_k$.
\item
$\rho_k\otimes \eta\cong \rho_{2^{n-2}-k}$.
\item
$\rho_k\otimes \epsilon\eta\cong \rho_{2^{n-2}-k}$.
\end{enumerate}
\end{prop}

\begin{proof}
   Follows from \cite[Theorem 4.4.3]{hirzebruch1966topological}.
\end{proof}
 \vspace{1ex}
\subsection{Cohomology of the Cyclic Groups}
Let $C_n$ denote the cyclic group of order $n$ with a generator $\alpha$. the cohomology ring of $C_n$ is well known.
$$H^*(C_n,\Z)=\frac{\Z[x]}{\langle nx\rangle},$$
where $x=c_1(\chi_1)$, $\chi_1(\alpha)=\zeta_n=e^{2\pi i/n}$ and $1\leq k\leq n-1$.

Let $G$ be a finite group and $G$ acts on $\mathbb{C}$ trivially. Then one has the isomorphism
$$\mathrm{Hom}(G,\C^{\times})\cong H^2(G,\Z).$$
For any character $\chi$ of $G$, the explicit isomorphism is given by $\chi\mapsto c_1(\chi)$, see \cite[Appendix (3)]{atiyah1961characters}, where $c_1$ denotes the first Chern class.

\section{Realizable Integral Cohomology Classes of Klein 4 group}\label{sec: klein 4 group}
Consider the Klein 4 group $V_4=C_2\times C_2$. We have 
$$H^*(V_4,\mathbb{Z})=\dfrac{\Z[x,y]}{\langle 2x,2y\rangle},$$
where $x=c_1(\sgn\otimes \mathbb{1})$ and $y=c_1(\mathbb{1}\otimes \sgn)$. Any representation $\pi$ of $V_4$ can be written as 
$$\pi=m\mathbb{1}\oplus a(\sgn\otimes\mathbb{1})\oplus b(\mathbb{1}\otimes \sgn)\oplus c(\sgn\otimes \sgn).$$
To write down the total Chern class of $\pi$ we use the fact $c_1(\varphi_1\otimes\varphi_2)=c_1(\varphi_1)+c_1(\varphi_2)$, where $\varphi_1$ and $\varphi_2$ are $1$-dimensional representations of $V_4$ (see \cite[Theorem 4.4.3]{hirzebruch1966topological}).

\begin{theorem}
We have $\mathcal{RC}(V_4)\neq \mathcal{H}^{\mathrm{even}}(V_4, \mathbb{Z})$.    
\end{theorem}

\begin{proof}
To prove this, we show that $x^3+y^3\notin \mathcal{RC}(V_4)$.  Let $\pi$ be an arbitrary representation of $V_4$. Without loss of generality, we can assume 
$$c(\pi)=(1+x)^a ((1+x)(1+y))^b (1+x+y)^c.$$
\begin{comment}
We have 
\begin{align*}
\left[((1+x)(1+y))^b (1+x+y)^c\right]_{3^{-}} =& 1 + (b+c)(x+y) +\binom {b+c}2(x^2+y^2)+ bxy\\
& + \binom {b+c}{3}(x^3 + y^3) +  \left(\binom {b+c}3 + bc\right)(x^2y + xy^2)
\end{align*}
\end{comment}
Collecting the third degree homogeneous terms we obtain 
\begin{align*}
c_3(\pi)=& \binom {a+b+c}{3}x^3+\binom{b+c}{3}y^3+\left(\binom a2 (b+c) + ab + \binom {b+c}3 + bc \right)x^2y\\
&+\left(a\binom{b+c}{2}+\binom {b+c}3 + bc\right)xy^2.    
\end{align*}
Suppose $c_3(\pi)=x^3+y^3$ for some representation $\pi$ of $V_4$. To satisfy this, we require the condition that the coefficients of $x^3$ and $y^3$ are odd, while the coefficients of $x^2y$ and $xy^2$ are even. If $\binom{a+b+c}{3}$ and $\binom{b+c}{3}$ both are odd, then elementary calculations show that both $a$ and $bc$ are even. But then the coeffient of $xy^2$ becomes odd. Therefore the class $x^3+y^3$ is not realizable.
\end{proof}

\section{Chern Classes of representations of $D_{2n}$}\label{sec: chern classes of d2n}

\vspace{1ex}
\subsection{The Group $D_{2n}$ with $n$ is odd}
~\vspace{1ex}

The integral cohomology ring for the Dihedral groups are given as follows: 

\begin{thm}\cite[Theorem 5.2]{handel1993products}\label{thm: cohomology of D_2n for n odd}
   Let $n\geq 3$ be odd. Then $$H^{*}(\dn, \Z)=\frac{\Z[a_2,d_4]}{\left( 2a_2,nd_4 \right)}$$
   where $a_2$ has degree $2$ and $d_4$ has degree $4.$
\end{thm}

From the description of the cohomology ring one obtains
$$H^2(D_{2n}, \mb Z)\cong C_2=\langle a_2\rangle \quad \text{and} \quad H^4(D_{2n}, \mb Z)=C_n\oplus C_2=\langle d_4 \rangle \oplus \langle a_2^2\rangle.$$

Consider the cyclic subgroups $R=\langle r  \rangle$ and $\Gamma=\langle s\rangle$ of $D_{2n}$. We write restriction maps at the cohomology level as $\iota_1^*:H^*(D_{2n})\to H^*(R)$ and $\iota_2^*:H^*(D_{2n}) \to H^*(\Gamma)$ where $\iota_1$ and $\iota_2$ denote the inclusions of the subgroups $R$ and $\Gamma$ respectively inside $D_{2n}$. Let $\alpha\in H^2(R,\mathbb{Z})$ denote the generator of  $H^*(R,\mathbb{Z})$ and write $\beta\in H^2(\Gamma,\mathbb{Z})$ for the generator of $H^*(\Gamma,\mathbb{Z})$.

\begin{lemma}\label{c1d2n}
We have 
$$c_1(\sigma_k)=c_1(\chi_s)=a_2.$$
\end{lemma}
\begin{proof}
For any representation $\pi$ of a group $G$, we know $c_1(\pi)=\det (\pi)$. Note that $\det (\sigma_k)=\chi_s$.    
\end{proof}

Consider the map 
$$\Phi_i:H^{2i}(D_{2n})\to H^{2i}(R)\oplus H^{2i}(\Gamma)\quad \text{defined as} \quad \Phi_i(\nu)=(\iota_1^*(\nu), \iota_2^*(\nu)).$$

\begin{lemma}\label{lem1}
For the representation $\sigma_k$ we have 
$$\Phi_2(c_2(\sigma_k))=(-k^2\alpha^2, 0).$$
\end{lemma}

\begin{proof}
Taking restriction of $\sigma_k$ to $R$ one obtains 
$\sigma_k\mid_{R}=\chi_k\oplus \chi_k^{-1},$
where $\chi_k=\zeta_n^k$. As a result
\begin{align*}
	c(\sigma_k\mid_{R})=c(\chi_k\oplus \chi_k^{-1})=c(\chi_k)\smile c(\chi_k^{-1})&=(1+c_1(\chi_k))(1+c_1(\chi_k{-1}))\\&=(1+k\alpha)(1-k\alpha)=1-k^2\alpha^2,
\end{align*}
where $\alpha=c_1(\xi_n)$.
In particular, $\iota_1^*(c_2(\sigma_k))=-k^2\alpha^2$.    On the other hand $\sigma_k\mid_{\Gamma}=\mathbb{1}\oplus \chi_s.$ This gives $$c(\sigma_k\mid_{\Gamma})=1+c_1(\chi_s)=1+a_2.$$
In particular, $\iota_2^*(c_2(\sigma_k))=0$.
\end{proof}

\begin{theorem}\label{thmisod2n}
	For $n$ odd, the map $\Phi_2:H^4(D_{2n})\to H^4(R)\oplus H^4(\Gamma)$ given by $\Phi_2(\nu)=(\iota_1^*(\nu), \iota_2^*(\nu))$ is an isomorphism.
\end{theorem}
\begin{proof}
Let $\epsilon_s$ denote the nontrivial irreducible representation of $\Gamma$. We compute
$$\iota_2^*(a_2)=\iota_2^*(c_1(\chi_s))=c_1(\epsilon_s)=\beta.$$
Suppose $c_2(\sigma_1)=md_4+\delta a_2^2$, where $m\in \Z/n\Z$ and $\delta\in \Z/2\Z$. Lemma \ref{lem1} asserts $\iota_2^*(c_2(\sigma_1))=0$. Then
\begin{align*}
\iota_2^*(md_4)=\delta\iota_2^*(a_2^2)&=\delta\beta^2.
\end{align*}
Since $\beta^2$ is an order $2$ element, we conclude that the order of $md_4$ is even. But $md_4\in C_n$ where $n$ is odd, which is a contradiction. So $\delta=0$, i.e. $c_2(\sigma_1)=md_4$.
From \Cref{lem1} we have
$$\iota_1^*(md_4)=\iota_1^*(c_2(\sigma_1))=-\alpha^2.$$
Since $-\alpha^2$ has order $n$, the order of $md_4$ is a multiple of $n$. Also $nd_4=0$. This shows that order of $md_4$ is equal to $n$. So $\gcd(m,n)=1$. Without loss of generality we can choose $m=1$. Therefore we have $c_2(\sigma_1)=d_4$ and 
\begin{equation}
 \iota_1^*(d_4)=\iota_1^*(c_2(\sigma_1))=-\alpha^2.   
\end{equation}
 
Since $n$ is odd, we have
$$H^4(D_{2n}, \mb Z)=\langle d_4 \rangle \oplus \langle a_2^2\rangle=C_n\oplus C_2\cong C_{2n}.$$
In particular, the element $d_4+a_2^2$ becomes a generator of $C_{2n}$. Evaluating $\Phi_2$ on $d_4+a_2^2$, we obtain
\begin{align*}
\Phi_2(d_4+a_2^2)&=(\iota_1^*(d_4+a_2^2), \iota_2^*(d_4+a_2^2))=(-\alpha^2,\beta^2).
\end{align*} 
Observe that $H^4(R)\oplus H^4(\Gamma)$ is also a cyclic group of order $2n$ with a generator $(-\alpha^2,\beta^2)$. Hence $\Phi_2$ is an isomorphism.
\end{proof}

\begin{theorem}
	For $n$ odd, the total Chern classes of non-trivial representations of $D_{2n}$ are given by
	\begin{enumerate}
\item 
$c(\chi_s)=1+a_2$.
\item
$c(\sigma_k)=1+a_2+k^2 d_4$.
\end{enumerate}
\end{theorem}
\begin{proof}
\begin{enumerate}
    \item The first claim follows directly from  \Cref{c1d2n}.
    
    \item From \Cref{lem1} we have $\Phi_2(c_2(\sigma_k))=(-k^2\alpha^2,0)$. Since $(-\alpha^2,\beta^2)$ is a generator of $H^4(R, \Z)\oplus H^4(\Gamma, \Z)$, we have 
$(-k^2\alpha^2,0)=\ell(-\alpha^2,\beta^2)$. This shows $\ell\equiv 0\pmod 2$ and $\ell\equiv k^2\pmod {n}$. By the Chinese remainder theorem we obtain $\ell=(n+1)k^2$.   Since $\Phi_2$ is an isomorphism and $\Phi_2(d_4+a_2^2)=(-\alpha^2,\beta^2)$ (see \Cref{thmisod2n}) we have $c_2(\sigma_k)=(n+1)k^2(d_4+a_2^2)$. As $nd_4=0$ and $n$ is odd, we have $c_2(\sigma_k)=k^2d_4$. 
\end{enumerate}
\end{proof}
\begin{corollary} If $n$ is an odd prime, then the total Chern class distinguishes every $2$-dimensional representation of $D_{2n}.$
\end{corollary}
\begin{proof}
    Suppose $c(\sigma_i)=c(\sigma_j),$ where $1\leq i \neq j \leq \frac{n-1}{2}.$ Then we have $i^2 \equiv j^2 (\text{~mod~} n).$ But when $n$ is an odd prime, there are exactly $\frac{n-1}{2}$ distinct quadratic residues modulo $n$. This forces $i=j.$ 
\end{proof}
However, for n composite, two representations might end up having the same total Chern class, which will be predominantly the case in the next section.

\vspace{1ex}
\subsection{The Group $D_{2n}$ with $n$ is even}~\vspace{1ex}

\begin{thm}\cite[Theorem 5.3]{handel1993products}\label{thm : cohomology of D_2n for n even}
    Let $n\geq 2$ be even. Then 
    $$H^{*}(\dn,\Z)=\frac{\Z [a,b,c,d]}{\left(2a,2b,2c,nd,b^2+ab+\frac{n^2d}{4}, c^2+ad\right)},$$
where $a$ and $b$ have degree $2$, $c$ has degree $3$ and $d$ has degree $4.$
\end{thm}

From the description of the cohomology ring we obtain
$$H^4(D_{2n}, \Z)=\langle a\rangle\oplus \langle b\rangle \cong C_2\oplus C_2,\quad H^4(D_{2n}, \Z)=\langle d \rangle \oplus\langle a^2 \rangle\oplus \langle b^2 \rangle =C_n\oplus C_2\oplus C_2.$$
 
Consider the subgroups $R:=\langle r\rangle, \Gamma:= \langle s \rangle, \Omega:=\langle rs  \rangle$
of $D_{2n}$. Let $\iota_1, \iota_2, \iota_3$ denote the embeddings of these subgroups in $D_{2n}$. Let $\alpha_r,\alpha_s,\alpha_{rs}$ denote the generators of the groups $H^2(R,\mathbb{Z}), H^2(\Gamma,\mathbb{Z}), H^2(\Omega,\mathbb{Z})$. Consider the ring homomorphism $$\Theta :  H^*(D_{2n},\Z) \to H^*(R,\Z) \oplus H^*(\Gamma,\Z) \oplus H^*(\Omega,\Z)$$  given by $\Theta_i(\nu)=(\iota_1^*(\nu), \iota_2^*(\nu), \iota_3^*(\nu))$, where the $i$-th component of $\Theta$ is  
$$\Theta_i: H^{2i}(D_{2n}, \Z)\to H^{2i}(R,\Z)\oplus H^{2i}(\Gamma, \Z)\oplus H^{2i}(\Omega), i\in \N$$

\begin{lemma}\label{lemma dihedral even}
We have the following equalities:
\begin{enumerate}
\item 
$\Theta_1(c_1(\chi_r))=(\frac{n}{2}\alpha_r, 0, \alpha_{rs})$.
\item 
$\Theta_1(c_1(\chi_s))=(0, \alpha_{s}, \alpha_{rs})$.
\item 
$\Theta_1(c_1(\chi_{rs}))=(\frac{n}{2}\alpha_r, \alpha_{s}, 0)$.
\item 
$\Theta_1(c_1(\sigma_k))=(0, \alpha_{s}, \alpha_{rs})$.
\item 
$\Theta_2(c_2(\sigma_k))=(-k^2\alpha_r^2, 0, 0)$.

\end{enumerate}
\end{lemma}
\begin{proof}

The equalities in $(1), (2), (3)$ can be easily obtained by restricting the corresponding characters to the subgroups $R, \Gamma, \Omega$ respectively. So we prove the remaining ones.
\end{proof}

\begin{lemma}\label{theta1inj}
The map 
$$\Theta_1: H^2(D_{2n}, \mb Z)\to H^2(R, \mb Z)\oplus H^2(\Gamma, \mb Z)\oplus H^2(\Omega, \mb Z),$$
given by $\Theta_1(\nu)=(\iota_1^*(\nu), \iota_2^*(\nu), \iota_3^*(\nu))$  is injective.
\end{lemma}

\begin{proof}
We compute the images of all nontrivial elements of $H^2(D_{2n}, \mb Z)$ under the map $\Theta_1$.
\begin{align*}
&\Theta_1(c_1(\chi_r))=\left(\frac{n}{2}\alpha_r, 0, \alpha_{rs}\right).\\
&\Theta_1(c_1(\chi_s))=\left(0, \alpha_{s}, \alpha_{rs}\right).\\
&\Theta_1(c_1(\chi_{rs}))=\left(\frac{n}{2}\alpha_r,\alpha_{s}, 0 \right).\\
\end{align*}
The map $\Theta_1$ is injective since all images are distinct.
\end{proof}

\begin{lemma}\label{theta2iso}
  Consider the map $$\Theta_2: H^4(D_{2n}, \Z)\to H^4(R, \Z)\oplus H^4(\Gamma,Z)\oplus H^4(\Omega,\Z), \quad \text{given by} \quad \Theta_2(\nu)=(\iota_1^*(\nu), \iota_2^*(\nu),\iota_3^*(\nu))$$
  Then $\Theta_2$ is an isomorphism.
\end{lemma}

\begin{proof}
We have

\begin{align*}
&\Theta_2(c_2(\sigma_1))=\left(-\alpha_r^2, 0, 0\right)\\
&\Theta_2(c_1(\chi_s)^2)=\left(0, \alpha_{s}^2, \alpha_{rs}^2\right)\\
&\Theta_2(c_1(\chi_{rs})^2)=\left(\left(\frac{n}{2}\alpha_r\right)^2,\alpha_{rs}^2, 0 \right)      \\
\end{align*}
This shows that the map $\Theta_2$ is surjective. Since $H^4(D_{2n},\Z)$ and $H^4(R,\Z)\oplus H^4(\Gamma,\Z)\oplus H^4(\Omega,\Z)$ have the same cardinality, it is an isomorphism.
\end{proof}

\begin{prop}
The elements $c_2(\sigma_1), a^2, b^2$ generate the cohomology group $H^4(D_{2n},\Z)$.
\end{prop}

\begin{proof}
Since $\Theta_2$ is an isomorphism and the elements $\Theta_2(c_2(\sigma_1)), \Theta_2(c_1(\chi_s)^2), \Theta_2(c_1(\chi_{rs})^2)$ generate the group $H^4(R, \Z)\oplus H^4(\Gamma,Z)\oplus H^4(\Omega,\Z)$. So the elements $c_2(\sigma_1), c_1(\chi_s)^2$ and $c_1(\chi_{rs})^2$ generate $H^4(D_{2n},\Z)$. Since $c_1(\chi_s)$ and $c_1(\chi_{rs})$ take the values $a, b$ or $a+b$, we conclude that $c_2(\sigma_1), a^2, b^2$ generate $H^4(D_{2n},\Z)$. 
\end{proof}

\vspace{1ex}
\subsubsection{The Case of $n\equiv 0\pmod 4$}~\vspace{1ex}
\begin{theorem}\label{tcn4}
The total Chern classes of nontrivial irreducible representations of $D_{2n}$ for $n\equiv 0\pmod 4$ are given by
\begin{align*}
 c(\chi_r)&= 1+a+b\\
 c(\chi_s)&=1+a\\
 c(\chi_{rs})&=1+b\\
 c(\sigma_k)&=1+a+k^2d'.
\end{align*}
\end{theorem}
\begin{proof}
Since $n\equiv 0\pmod 4$, the relation $b^2+ab+\frac{n^2d}{4}=0$ in the cohomology ring reduces to $b(a+b)=0$. Observe that $\Theta_1(c_1(\chi_r)c_1(\chi_{rs}))=0$. Since $\Theta_1$ is injective we have $c_1(\chi_r)c_1(\chi_{rs})=0$. Therefore, we can assume $a=c_1(\chi_s)$ and $b=c_1(\chi_{rs})$.
One computes $\Theta_1(c_1(\sigma_k))=\Theta_1(c_1(\chi_s))$. Since $\Theta_1$ is injective  we have $c_1(\sigma_k)=c_1(\chi_s)=a$.    
\end{proof}
~\vspace{1ex}
\subsubsection{The Case of $n\equiv 2\pmod 4$}~\vspace{1ex}
 Next we compute the total Chern classes for irreducible representations of $D_{2n}$ for $n\equiv 2\pmod 4$.
\begin{theorem}
The total Chern classes of nontrivial irreducible representations of $D_{2n}$ for $n\equiv 2\pmod 4$ are given by
\begin{align*}
 c(\chi_r)&= 1+a+b\\
 c(\chi_s)&=1+a\\
 c(\chi_{rs})&=1+b\\
 c(\sigma_k)&=1+a+k^2d'.
\end{align*}
\end{theorem}
\begin{proof}
We have $\sigma_k \otimes \chi_r \cong \sigma_{n/2-k}.$
Moreover
$$c_2(\sigma_k \otimes \chi_r)= c_2(\sigma_k) + c_1(\sigma_k)c_1(\chi_r) + (c_1(\chi_r))^2. $$ 
On the other hand 
\begin{align*}
c_2(\sigma_{n/2-k})&=(n/2-k)^2 d'\\
&= \left((n/2)^2-nk +k^2\right) d'\\
&=\left(n^2/4 + k^2 \right) d'    
\end{align*}
Equating both equations, we get 

$\left((n/4)^2+k^2\right) d'= k^2 d' +c_1(\sigma_k)c_1(\chi_r) + (c_1(\chi_r))^2.$
Also, $c_1(\sigma_k)= c_1(\chi_s)$. So we obtain $$c_1(\chi_r)\left( c_1(\chi_s) + c_1(\chi_r) \right)= (n^2/4) d' . $$

Applying $\Theta$ to both sides obtains

$$\Theta_1(c_1(\chi_r)) \Theta_1(c_1(\chi_r) + c_1(\chi_s))= ((n^2/4) \alpha_r^2, 0,0). $$
Note that for the $n \equiv 2 \pmod 4$ case, we have $c_1(\chi_r) c_1(\chi_{rs})= \frac{n^2}{4}d'$ and the cohomology ring relation $b^2 +ab+ \frac{n^2}{4}d=0.$
Put $a'=a+b$. The relation becomes $a'b=\frac{n^2}{4}$.

Comparing both these we conclude that the following choices 
\begin{itemize}
    \item $\{ c_1(\chi_r), c_1(\chi_{rs}) \} = \{ b, a'\} $ and $d'=d$. Since, by \Cref{theta1inj} and \Cref{theta2iso}, we have $\Theta(c_1(\chi_r)c_1(\chi_{rs}))=(\frac{n^2}{4},0,0)$. Hence, $c_1(\chi_r)c_1(\chi_{rs})= \Theta^{-1}((\frac{n^2}{4},0,0)) = \frac{n^2}{4}d'$. Thus $d=d'.$
    \item  $\{c_1(\chi_s), c_1(\chi_{rs}) \}= \{b,a'\}$ and $\Theta^{-1}(2n',1,0)= d.$ Since, by \Cref{theta1inj} and \Cref{theta2iso}, we have $\Theta(c_1(\chi_s)c_1(\chi_{rs}))=(0,1,0)$. Hence, $c_1(\chi_s)c_1(\chi_{rs})= \Theta^{-1}((0,1,0)) = \frac{n^2}{4}\Theta^{-1}((2n',1,0)),$ where $n'$ is coprime to $n.$ Note that, $\Theta^{-1}((2n',1,0))$ is of order $n.$
    \item $\{ c_1(\chi_{r}), c_1(\chi_{s})\} = \{b,a'\}$ and $\Theta^{-1}(2n',0,1)=d$ Since, by \Cref{theta1inj} and \Cref{theta2iso}, we have $\Theta(c_1(\chi_r)c_1(\chi_{s}))=(0,0,1)$. Hence, $c_1(\chi_s)c_1(\chi_{rs})= \Theta^{-1}((0,0,1)) = \frac{n^2}{4}\Theta^{-1}((2n',0,1),$ where $n'$ is coprime to $n.$ Note that, $\Theta^{-1}((2n',0,1))$ is of order $n.$
\end{itemize}
Hence, we can assign $c_1(\chi_s)=a, c_1(\chi_{rs})=b$ and $c_1(\chi_{r})=a+b.$
\end{proof}
\section{Realizable Integral Cohomology Classes of $D_{2n}$}\label{sec: realizability of chern classes of D2n}
\begin{lemma}\label{surjective}
Consider the restriction map $\Psi: H_C^*(D_{2km})\to H_C^*(D_{2k})$, where $m$ is a positive integer. 
\begin{enumerate}
\item 
If $k$ is odd, then  $\Psi$ is surjective for any positive integer $m$.
\item 
If $k$ is even, then $\Psi$ is surjective if $m$ is odd.
\end{enumerate}
\end{lemma}
\begin{proof}
We denote the generators of $D_{2n}$ by $r_n$ and $s$. We can realize $D_{2k}$ as a subgroup of $D_{2km}$ by sending $ r_{k} \mapsto r_{km}^m, s \mapsto s$. For $n$ even, we write the representations of $D_{2n}$ as $\chi_{r,n}, \chi_s, \chi_{rs,n}, \sigma_{k,n}$.

Furthermore, when $n$ is odd, we write $a_{2,n}$ and $d_{4,n}$ for the cohomology classes $a_2$ and $d_4\in H^*(D_{2n}, \Z)$ respectively. When $n$ is even,  we additionally denote the class $b_2\in H^*(D_{2n}, \Z)$ by $b_{2,n}$.

The ring $H^*_C(D_{2n,\Z})$ is generated by the Chern classes of irreducible reprsentations of $D_{2n}$. So we have
\[
H_C^*(D_{2n},\Z)=
\begin{cases}
 \langle a_2, d_4'\rangle, & \text{when $n$ is odd},\\
 \langle a_2,b_2, d_4'\rangle & \text{when $n$ is even}.
\end{cases}
\]
In both cases recall that $d_4'=c_2(\sigma_1)$. 

Observe that $\sigma_{1,km}\mid_{D_{2k}}=\sigma_{1,k}$. Therefore $\Psi(d'_{4,km})=d'_{4,m}$. Also $\Psi(a_{2,km})=a_{2,k}$.
For $k$ odd, $H_C^*(D_{2k},\Z)$ is generated by $a_{2,k}$ and $d_{4,k}$. So we conclude that $\Psi$ is surjective.

Now consider the case where $k$ is even. We compute
$$\chi_{r,km}(r_{km}^m)=(-1)^m.$$
This shows that $\chi_{r,km}\mid_{D_{2k}}=\chi_{r,k}$ if $m$ is odd, otherwise the restriction is trivial. Similarly $\chi_{rs,km}\mid_{D_{2k}}=\chi_{rs,k}$ if $m$ is odd, otherwise the restriction gives $\chi_{s,k}$. So if $m$ is even,  $\Psi(b_{2,km})=a_{2,k}$ and $b_{2,k}$ has no preimage. Therefore $\Psi$ is not surjective. On the other hand, for $m$ odd, $\Psi(b_{2,km})=b_{2,k}$. Therefore $\Psi$ is surjective.   
 \end{proof}
\begin{theorem}\begin{enumerate}
\item 
For $k$ is odd, if $\mathcal{RC}(D_{2k})\neq \mathcal{H}^{\text{even}}(D_{2k})$, then $\mathcal{RC}(D_{2km})\neq H^{\text{even}}(D_{2km})$ where $m$ is any positive integer.
\item 
For $k$ is even, if $\mathcal{RC}(D_{2k})\neq \mathcal{H}^{\text{even}}(D_{2k})$, then $\mathcal{RC}(D_{2km})\neq \mathcal{H}^{\text{even}}(D_{2km})$ where $m$ is an odd integer.
\end{enumerate}    
\end{theorem}
\begin{proof}
    These follow from \Cref{surjective}.
\end{proof}
\vspace{1ex}
\subsection{The Realizable Classes $\mathcal{RC}(D_{2n})$ with $n$ odd}~~\vspace{1ex}
Any representation $\rho$ of $D_{2n}$ can be written as $\rho= x \mathbb{1} \oplus y \chi_s \oplus \bigoplus\limits_{k} z_k \sigma_k.$ 
We note an immediate observation about the $4k+2$-nd cohomology group for any $D_{2n}$, $n$ is odd.

\begin{theorem}
    Any cohomology class in $H^{4k+2}(D_{2n}, \mb Z)$ is realizable, where $n$ is odd.
\end{theorem}
\begin{proof}
    Since $a_2^{2k+1}$ is the generator of $H^{4k+2}(D_{2n}, \mb Z) \cong \mb Z_2$ and $c_{2k+1}(\chi_s^{2k+1})=a_2^{2k+1},$ the result follows. 
\end{proof}
However it is not true that all cohomology classes of $D_{2n}$, for $n$ odd, are realizable. Here is one such instance.
\begin{example}\label{ex: D_6 example}
\normalfont{
    Our claim is that for $n=3$, the class $2d_4^2 \in H^{8}(D_6, \mb Z)$ cannot appear as a Chern class. This is because the binomial coefficients $\binom{z_k}{2}$ modulo $3$ never takes the value $2.$ Indeed, if there  exists some $z_k$ such that $ \binom{z_k}{2} \equiv 2 \pmod 3$ then $z_k(z_k-1) \equiv 1 \pmod 3.$ But then $(z_k+1)^2 \equiv 2 \pmod 3$, which is a contradiction as the only quadratic residues modulo $3$ are $0$ and $1$. }
\end{example}

 Let $\psi: G \to H$ be a group homomorphism between two groups $G$ and $H$. Let $\psi^*: H^n(H)  \to H^n(G)$ be the induced map at the group cohomology level. Let $\psi^*(c_H)= c_G$, where $c_H$ and $c_G$ be cohomology classes.
\begin{lemma} \label{realize}
	If $c_G \notin \mathcal{RC}(G)$ then $c_H \notin \mathcal{RC}(H).$
\end{lemma}
\begin{proof}
	If $c_H \in \mathcal{RC}(H)$, then there exists a representation $\pi$ of $H$ such that $c_H = c_n(\pi)$. But then $c_G = \psi^*(c_H)= \psi^*(c_n(\pi)) = c_n(\pi \circ \psi) \in \mathcal{RC}(G).$ This is a contradiction. Hence $c_H \notin \mathcal{RC}(H)$.
\end{proof}
 
\begin{corollary}\label{cor: D_6m, m odd}
	There is always a cohomology class in $H^*(D_{6m}, \Z)$ which is not realized by any representation. 
\end{corollary}
\begin{proof}
By \Cref{ex: D_6 example}, $2d^2 \notin \mathcal{RC}(D_6).$ The remaining part follows from \Cref{realize} and \Cref{surjective}.	
\end{proof}

\begin{theorem}\label{thm: D10 realizable}
Every cohomology class in $H^{4k}(D_{10}, \mb Z)$ is realizable as a Chern class $c_{2k}$ of some representation.
\end{theorem}
\begin{proof}
We now focus on $D_{10}$ and inspect whether $a_2^{2k}+md_4^k$ can appear as a Chern class $c_{2k}$ for $m \in \mb Z_5$. Observe that $a_2 \smile d_4 =0$, since $a_2$ and $d_4$ have coprime orders. for any representation $\rho$ of $D_{10}$, the total Chern class $$c(\rho) =\left(1+a_2\right)^y\left(1+a_2 + d_4\right)^{z_1}\left(1+a_2 + 4d_4\right)^{z_2}.$$ The coefficient of $a_2^{2k}$ is $\binom{y+z_1+z_2}{2k} \pmod 2$ and the coefficient of $d_4^k$ is $$\binom{z_1}{k}-\binom{z_1}{k-1}z_2+\binom{z_1}{k-2}\binom{z_2}{2}- \dots + (-1)^{k-1}z_1 \binom{z_2}{k-1} +(-1)^k\binom{z_2}{k} \pmod 5.$$
We first look at the coefficient of $d_4$ and see if they cover all possible nonzero residues modulo $5.$ 
We summarize the computations for different cases in the following table. Let $c_d$ denote the coefficient of $d_4^k$.
\[
\begin{tabular}{ |c|c|c| }
 \hline
$k \pmod 5$ & $(z_1, z_2)$ & $c_d\pmod 5$ \\
\hline
\multirow{2}{4em} {$0,1,2,3,4$} & $(k,0)$ & $1$  \\
\cline{2-3}
& $(k-1,1)$ & $4$ \\
\hline
\multirow{2}{4em} {$1$} & $(k+2,1)$ & $2$  \\
\cline{2-3}
& $(k+2,0)$ & $3$ \\
\hline
\multirow{2}{4em} {$2$} & $(k+2,1)$ & $2$  \\
\cline{2-3}
& $(k,2)$ & $3$ \\
\hline
\multirow{2}{4em} {$3$} & $(k+1,3)$ & $2$  \\
\cline{2-3}
& $(k,1)$ & $3$ \\
\hline
\multirow{2}{4em} {$4$} & $(k,1)$ & $2$  \\
\cline{2-3}
& $(k-1,3)$ & $3$ \\
\hline
\end{tabular}
\]
 
It only remains to show that we obtain the residue classes $2$ and $3$ for the case $k \equiv 0 \pmod5$.
Observe that $$C=\binom{z_1}{k}-\binom{z_1}{k-1}z_2+\binom{z_1}{k-2}\binom{z_2}{2}- \dots + (-1)^{k-1}z_1 \binom{z_2}{k-1} +(-1)^k\binom{z_2}{k}$$ is the coefficient of $x^k$ in the expression $(1+x)^{z_1}(1-x)^{z_2}.$
When $k \not\equiv 0 \pmod 5$, we always have solutions as described above.  
Let $k=5^i \cdot m$ where $\gcd(5,m)=1$. Then we have solutions $(\alpha_1, \alpha_2)$ integers such that $\alpha_1+\alpha_2 \geq m$ and the coefficient of $x^m$ in $(1+x)^{\alpha_1}(1-x)^{\alpha_2}$ is congruent to $2 \text{ or } 3 \pmod 5$. If we put $(z_1, z_2) = (5^i\alpha_1, 5^i \alpha_2)$, then the coefficient of $x^{5^i m=k}\,\,$in $(1+x)^{z_1}(1-x)^{z_2} \pmod 5$ is
\begin{align*}
    (1+x)^{z_1}(1-x)^{z_2}&= (1+x)^{5^i \alpha_1}(1-x)^{5^i \alpha_2}\\
    &=(1+x^{5^i})^{\alpha_1}(1-x^{5^i})^{\alpha_2}
\end{align*}
Now put $y=x^{5^i}$, then we have $$(1+x)^{z_1}(1-x)^{z_2}=(1+y)^{\alpha_1}(1-y)^{\alpha_2}$$
So the coefficient $C$, which is the coefficient of $x^{5^i m}$,  is the coefficient of $y^m$, in the above expression. Thus we can choose it to be $2$ or $3$ as we like. 

The coefficient of $a_2^{2k}$ is $\binom{y+z_1+z_2}{2k}$. For chosen of $(z_1, z_2)$ and $k$ we need to show that $\binom{y+z_1+z_2}{2k}$ attains both odd and even values. For that we take
$$n=\mathrm{max}\{\log_2 (z_1+z_2), \log_2 2k\}+1.$$
Then $2k< 2^n$. One easily checks using Lucas's theorem that
\[
\binom{y+z_1+z_2}{2k}=
\begin{cases}
0\pmod 2, & \text{if $y=2^n-z_1-z_2$},\\    
1\pmod 2, & \text{if $y=2^n-z_1-z_2-1$}.
\end{cases}
\]
\end{proof}
As it turns out the realization question is tricky for general $D_{2n}$. However, we offer the following consequence:
\begin{theorem}
    Let $p$ be an odd prime and $q$ is a power of $p$. Then every cohomology class in $H^{4q}(D_{2p},\mb Z)$ is realizable as a Chern class $c_{2q}$ of some representation of $D_{2p}$.
\end{theorem}
\begin{proof}
The coefficient of $a_2^{2q}$ and $d_4^q$ in $\left(1+a_2\right)^y  \left(1+a_2 +d_4\right)^{z_1}$ are $\binom{y+z_1}{2q}$ and $\binom{z_1}{q}$. Choose $z_1=rq$ where $r \in \mb Z_p$ and $y= 2^M+2q-rq$ with $M> \nu_p(q)$. Then $\binom{z_1}{q} \equiv \binom{r}{1} \pmod p$ and $\binom{y+rq}{2q}$ odd, by Lucas's Theorem. 
\end{proof}
\begin{remark}
    In the last theorem, one can actually take $q$ to be an integer whose base $p$ expansion has at least one $1.$
\end{remark}
 \subsection{The Realizable Classes $\mathcal{RC}(D_{2n})$ with $n$ even}~~\vspace{1ex}

 We change the cohomology ring generators of $D_{2n}$ with $n$ even (see \ref{thm : cohomology of D_2n for n even})
 to $\{a', b, c, d'\}$, where $a'=a+b$ and $d'=c_2(\sigma_1)$. As a result, the presentation becomes
 $$H^{*}(\dn,\Z)=\frac{\Z [a',b,c,d']}{\left\langle2a',2b,2c,nd,ba'+\frac{n^2d'}{4}, c^2+ad'\right\rangle}.$$
 
 In particular, we have
 $$H^6(D_8,\Z)=\langle a'^3 \rangle\oplus \langle b^3\rangle\oplus \langle a'd'\rangle \oplus \langle bd'\rangle\cong C_2^4.$$

 \begin{lemma}\label{h6}
     The classes $a'^3 + a'd'$ and $b^3+ bd'$ in $H^6(D_8, \Z)$ are not realizable.
 \end{lemma}
 \begin{proof}

Any arbitrary representation $\pi$ of $D_8$ is of the form
$$\pi=m_0\mathbb{1}\oplus m_1\chi_r\oplus m_2\chi_s\oplus m_3\chi_{rs}\oplus m_4\sigma_1.$$
Using \Cref{tcn4} we obtain the total Chern class of $\pi$ as
\begin{equation}
c(\pi)=(1+a')^{m_1}(1+a'+b)^{m_2}(1+b)^{m_3} (1+a'+b+d')^{m_4}.
\end{equation} 

We first compute $c_3(\pi)$.
Note that the element $b$ appears in exactly $m_2+m_3+m_4$ factors in the expression, namely $m_2$ factors $(1+a'+b)$, $m_3$ factors $(1+b)$ and $m_4$ factors $(1+a'+b+d')$. Therefore, the coefficient of $b^3$ is $\binom{m_2+m_3+m_4}{3}$. Similarly, one computes the coefficient of $a'^3$ to be $\binom{m_1+m_2+m_4}{3}$.

To compute the coefficient of $bd'$, observe that the occurrence of $d'$ must come from one of the $m_4$ factors of $(1+a'+b+d')$. After choosing a factor $(1+a'+b+d')$, the term $b$ can be chosen from a different factor that contains $b$. There are $m_2$ factors $(1+a'+b)$, $m_3$ factors $(1+b)$, and the remaining $m_4-1$ factors $(1+a'+b+d')$. Thus, for each of the $m_4$ choices of the factor contributing $d'$, there are $m_2+m_3+m_4-1$ choices for the factor contributing $b$.
So the coefficient of $bd'$ is $m_4(m_2+m_3+m_4-1)$. An identical argument shows that the coefficient of $a'd'$ is $m_4(m_1+m_2+m_4-1)$.
\begin{align*}
c_3(\pi)=& \binom{m_1+m_2+m_4}{3}a'^3+\binom{m_2+m_3+m_4}{3}b^3\\
& +m_4(m_1+m_2+m_4-1)a'd'+m_4(m_2+m_3+m_4-1)bd'.    
\end{align*}

     If it were, we must have $\binom{x+w+y}{3}$ even, $\binom{x+w+z}{3}$ odd, $w$ odd, $(x+w+y)$ odd and $(x+w+z)$ even. However, it is not possible that an integer $\beta$ is even but $\binom{\beta}{3}$ is odd.
 \end{proof}
We end the section by an observation analogous to Corollary \ref{cor: D_6m, m odd}.

\begin{corollary}\label{thm: non realizable class for D8k, k odd}
    There exists a cohomology class in $H^6(D_{8k}, \Z)$ for all odd positive integers $k$, which is not realisable.
\end{corollary}

\begin{proof}
Since the restriction map $\Psi:H^*(D_{8k},\Z)\to H^*(D_8, \Z)$ is surjective, when $k$ is odd (see \Cref{surjective}), the proof follows from \Cref{h6} and \Cref{realize}.    
\end{proof}

\section{Chern Classes of representations of $Q_{2^n}$}\label{sec: Chern classes for Q2n}
Let $\sigma_k : D_{2^{n-1}} \to O(2)$ be a two degree representation. Consider the double cover $\mathrm{Pin}^{-}(2)$ of $O(2)$ (see Section $13.1$ "Central Extension" in \cite{joshi2026characteristic}). The pullback of this covering map under $\sigma_k$ turns out to be $Q_{2^n}$, the generalized quaternion group, whose all $2^{n-2}-1$ many $2$-dimensional irreducible representatitions $\rho_k$ are induced from the one dimensional representations of the cyclic subgroup of $C_{2^{n-1}},$ which is a double cover of the cyclic subgroup $C_{2^{n-2}}$ of $D_{2^{n-1}}.$ Now, $\rho_k$ descends to $D_{2^{n-1}}$ if and only if $k$ is even. Consequently, $2^{n-3} -1$ many $2$-dimensional irreducible representations of $Q_{2^n}$ descend to $D_{2^{n-1}}.$ The remaining $2^{n-3}$ many of these $\rho_k$ are genuine representations of $Q_{2^n}$, which do not descent to $D_{2^{n-1}}.$  The diagram below captures the relation between the dihedral groups and generalized quaternion groups.
\[
\begin{tikzcd}
C_{2^{n-1}} \arrow[r, hookrightarrow] & Q_{2^n} \arrow[d, "q"] \arrow[r] & \mathrm{Pin}^{-}(2) \arrow[d, "p"]\\
C_{2^{n-2}} \arrow[r, hookrightarrow] & D_{2^{n-1}} \arrow[r, "\sigma_1"] & O(2)
\end{tikzcd}
\]
Consider the induced map at the even cohomology level $$q^*: H_{even}^*(D_{2^{n-1}},\Z) \longrightarrow H_{even}^*(Q_{2^n}, \Z)$$
By the analysis in the last section, $H_{even}^*(D_{2^{n-1}},\Z)$ is generated by the Chern classes $c_2(\sigma_1), c_1(\chi_r)$ and $c_1(\chi_{rs}).$ Hence, by naturality of characteristic classes, one can determine the Chern classes of $\rho_k$ for $k$ even. However, this is not enough to determine the Chern classes of the remaining ``Spin'' representations. This section is devoted for that purpose. We start by recalling the cohomology ring of the Quaternion group $Q_{2^n}$:

\begin{theorem}\cite[Theorem 3]{hayami2002cohomology}\label{thm: cohomology of Q_2^n}
We have 
     $$H^*(Q_{2^n},\Z)=  \dfrac{\Z[a,b,c]}{\langle2a, 2b, 2^nc,a^2, b^2-2^{n-1}c, ab-2^{n-1}c \rangle}$$
where degrees of $a$ and $b$ are $2$ and degree of $c$ is $4$.     
\end{theorem}
 \begin{theorem}
For the representations $\rho_k$ of $Q_{2^n}$ we have
 
$$
c_1(\rho_k)=
\begin{cases}
0, & \text{if $k$ is odd}\\
c_1(\epsilon), & \text{if $k$ is even}.
\end{cases}
$$

\end{theorem}
\begin{proof}
We have the isomorphism
$$c_1: H^2(Q_{2^n}, \mb Z) \to \mathrm{Hom}( Q_{2^n} , \mb C^{\times}),$$
where $c_1$ denotes the first Chern class. So we have $c_1(\rho_k) = c_1 ( \mathrm{det}(\rho_k))$. For $k$  odd, we obtain $\det(\sigma_k)=\mathbb{1}$. On the other hand, for $k$ even, $\det(\sigma_k)=\epsilon$. 
\end{proof}
Now we state a result quantifying the torsion carried by the Chern Classes and adapt \cite[Theorem V.1.8]{adem2013cohomology} as follows:
\begin{theorem}\label{exp}
    For a finite group $G$, let $\phi: G \to SU(2)$ be a faithful representation, then $|G|$ divides $\ord(c_2(\phi)).$
\end{theorem}
\begin{proof}
    Consider the Borel Construction on $SU(2)$, arising due to the faithful representation $\phi : G \to SU(2)$ yielding a free action of $G$ on $SU(2)$: 
    \[
    \begin{tikzcd}
       E: &SU(2) \arrow[r]\arrow[d] &  (SU(2) \times EG)/G \arrow[r]\arrow[d] & BG \arrow[d, "B\phi"]\\
       {}^GE: & SU(2) \arrow[r] & ESU(2) \arrow[r] & BSU(2)
    \end{tikzcd}
    \]
   The associated spectral sequence to the first row has $$E_2^{p,q}= H^p(G, H^q(SU(2))) \implies H^{p+q}(SU(2)/G)$$ 
    Since, $H^*(SU(2), \Z) = \Lambda[x_3]$. Denote by ${}^GE$ the spectral sequence associated to the universal $G$-bundle, which is the second row in the above diagram. By construction, the class $x_{3} \in {}^GE_2^{0,3}$ transgresses to the universal Chern class $c_2 \in {}^GE_{2}^{4,0}$ under the differential map ${}^Gd_4$. By naturality of spectral sequences, we have,
    $$d_4(x_3) = \phi^*({}^G d_4(x_3))= \phi^*(c_2) = c_2(\phi)$$
    Since $d_{4}\left(\ord(c_2(\phi)).x_{3}\right)=\ord\left(c_2(\phi)\right).c_2(\phi)$ and by definition $\exp\left(c_2(\phi)\right).c_2(\phi)=0$, we have $\ord\left(c_2(\phi)\right).x_{3}$ is a permanent cocycle surviving to the $E_{\infty}$ page. Consequently, $\ord(c_2(\phi)).x_{3} \in \mathrm{im}(\pi^*)$, where $\pi: SU(2) \to SU(2)/G$ is the covering map of oriented manifolds of degree $|G|$. Hence the image of $\pi^*$ on $H^{4}(SU(2)/G, \Z)$ is generated by $|G|.x_3$. Thus, the result follows. 
\end{proof}

We use this theorem to prove the following result.
 
 \begin{lemma}\label{lem: e is generator of H^4}
 The top  Chern class $c_2(\rho_1)$ is a generator of the cohomology group $H^4(Q_{2^n}, \mb Z)$.
\end{lemma}

\begin{proof}
The representation $\rho_1$ of $Q_{2^n}$ is faithful. So by Theorem \ref{exp} we have $2^n$ divides $ \ord(c_2(\rho_1))$. Since $H^4(Q_{2^n}, \mb Z)\cong C_{2^n}$, we conclude that the order of $c_2(\rho_1)$ is $2^n$. In other words, $c_2(\rho_1)$ is a generator of $H^4(Q_{2^n}, \mb Z)$.
\end{proof}
Next, we compute the second Chern classes of the representations $\rho_k$. 
\begin{theorem}
We have   
$$c_2(\rho_k)=k^2c_2(\rho_1).$$
\end{theorem}
\begin{proof}
Consider the subgroup $\Lambda :=\langle x  \rangle$ of $Q_{2^n}$. Note that $\Lambda\cong C_{2^{n-1}}$. Let $\iota: \Lambda \hookrightarrow Q_{2^n}$ denote the inclusion map. This induces a restriction map  $\iota^*: H^4(Q_{2^n}, \mb Z) \to H^4(\Lambda, \mb Z)$. The cohomology group $H^4(\Lambda,\mb Z)$ is cyclic of order $2^{n-1}$. Let $\beta^2$ be the generator of $H^4(\Lambda,\mb Z)$, where $\beta=c_1(\chi)$.
Restricting $\rho_r$ to $\Lambda$ gives 
$$ \rho_k\circ\iota= \chi^{\otimes k} \oplus (\chi^{-1})^{\otimes k},$$
where $\chi$ is a character of $\Lambda$ given by $\chi(x)=\zeta_{2^{n-1}}$.  
We compute
\begin{align*}
    \iota^*c(\rho_k) &= c(\chi^{\otimes k}) \smile c((\chi^{-1})^{\otimes k})\\
    &=(1+kc_1(\chi) )\smile (1+kc_1(\chi^{-1}))\\
    &=(1+k\beta) \smile (1-k\beta)\\
   & =1-k^2\beta^2.
\end{align*}
Hence, $\iota^*c_2(\rho_k)= -k^2\beta^2.$ 
\begin{comment}
But top Chern class is the Euler class $e$ of the representation $\rho_1$ and we obtain $\mathrm{res}_{C_{2^{n-1}}}^{Q_{2^n}} e = -\beta^2.$    
\end{comment}
 We write $e=c_2(\rho_1)$. 
By \Cref{lem: e is generator of H^4}, the element $e$ is the generator of $H^4(Q_{2^n}, \mb Z)\cong C_{2^n}.$ However, $-\beta^2$ is a generator of $H^4(\Lambda, \mb Z)\cong C_{2^{n-1}}.$ Also we have $\iota^*e=-\beta^2$. Consequently, $\mathrm{Ker} (\iota^*) = \{0, 2^{n-1} e\}.$ Therefore either $c_2(\rho_k)=k^2e$ or $c_2(\rho_1)=k^2e + 2^{n-1}e,$ for $k=1, \dots, 2^{n-2}-1.$ 

To resolve this, we consider $Q_{2^n}$ as a subgroup of $Q_{2^{n+1}}$. In terms of generators we have
$$Q_{2^n}=\langle x^2,y\mid x^{2^{n}}=y^4=1, x^{2^{n-1}}=y^2, y^{-1}xy=x^{-1}  \rangle\leq Q_{2^{n+1}}.$$
The corresponding map in the cohomology level is given by 
$$\mathrm{res}^*:H^4(Q_{2^{n+1}}, \mb Z)   \to H^4(Q_{2^n}, \mb Z). $$
For a fixed $k$, we write $\rho_{k,m}$ to denote the $2$-dimensional representation $\rho_k$ of the group $Q_{2^m}$ and $e_m$ for $c_2(\rho_{1,m})$. Note that the representation $\rho_{k, n+1}$ of $Q_{2^{n+1}}$ restricts to the representation $\rho_{k,n}$ of $Q_{2^n}.$ Therefore,
\begin{equation}
\mathrm{res}^*(e_{n+1})=c_2(\rho_{k,n+1}\mid_{Q_{2^n}})=c_2(\rho_{k,n})=e_n.    
\end{equation}
Recall that we had two choices for $c_2(\rho_{k,n+1})$, namely $e_{n+1}k^2$ or $e_{n+1}k^2+ 2^n e_{n+1}.$ But we have
$$\mathrm{res}^*(e_{n+1} k^2 + 2^n e_{n+1})= k^2\mathrm{res}^* (e_{n+1}) + 2^n\mathrm{res}^*(e_{n+1})= k^2e_n=\mathrm{\res}^*(e_{n+1} k^2).$$
So we conclude that $c_2(\rho_{k,n} )= k^2e_n$ for all $n.$  Thus in the original notation, we have $c_2(\rho_k)=k^2e$.   
\end{proof}

As an immediate consequence, we have the following result.

\begin{corollary}
For $k^2 \equiv 0 \pmod {2^n}$, we have $c_2(\rho_{k})=0.$
\end{corollary}

It remains to explicitly determine the first Chern classes of irreducible representations of $Q_{2^n}$. 

\begin{theorem}\label{thm: total CC of Q_2^n}
    The total Chern classes of all irreducible representations of $Q_{2^n}$ are given as follows:
\vskip 2mm
\begin{enumerate}
\item 
$c(\epsilon)=1+a ; \c(\eta)=1+b ; c(\epsilon\eta)=1+a+b;$
\vskip 3mm
\item 
$c(\rho_k) =
\begin{cases}
1+k^2e, & \text{for $k$ odd}\\  
1+a+ k^2e, & \quad \text{for $k$ even}.
\end{cases}
 $
\end{enumerate}
where $a$ and $b$ are the degree 2 generators of $H^*(Q_{2^n}, \mb Z)$ and $e=c_2(\rho_1)$.
\end{theorem}

\begin{proof}
We have the isomorphism
$$c_1: H^2(Q_{2^n}, \mb Z) \to \mathrm{Hom}( Q_{2^n} , \mb C^{\times}),$$
where $c_1$ denotes the first Chern class. From the description of the cohomology ring of $Q_{2^n}$ we obtain $H^2(Q_{2^n}, \mb Z)=\langle a, b\rangle$. In particular, the elements of $H^2(Q_{2^n}, \mb Z)$ are $\{1,a,b,a+b\}.$
On the other hand the elements of $\mathrm{Hom}( Q_{2^n} , \mb C^{\times})$ are $\{\mathbb{1}, \eta, \epsilon, \eta\epsilon\}$. Since $c_1$ is an isomorphism $c_1(\eta)$ and $c_1(\epsilon)$ must be different. 

From \cite[Theorem 4.4.3]{hirzebruch1966topological}, we obtain
\begin{equation}\label{tensor}
c_2( \rho_k \otimes \eta)= c_2(\rho_k)+ c_1(\rho_k)c_1(\eta)+c_1(\eta)^2.    
\end{equation}

For $n=3$, there is only one irreducible degree $2$ representation, namely $\rho_1$. We know that $c_1(\rho_1)=0$. For $n\geq 4$, we have
$$\rho_k \otimes \eta\cong \rho_{2^{n-2}-k}.$$
For $k$ even and $n\geq 4$, one computes 
$$(2^{n-2}-k)^2=2^{2n-4}-k2^{n-1}+k^2\equiv k^2\pmod {2^n}.$$
Therefore 
$$c_2(\rho_{2^{n-2}-k})=k^2e=c_2(\rho_k).$$
Using Equation \eqref{tensor} and the fact that $c_1(\rho_k)=c_1(\epsilon)$ we obtain
$$c_1(\eta)(c_1(\epsilon)+c_1(\eta))=0.$$
Since the cohomology ring has the relation $b(a+b)=0$. So without loss of generality, we can assume $c_1(\eta)=b$ and $c_1(\epsilon)+c_1(\eta)=a+b$. So we have $c_1(\epsilon)=a$.
\end{proof}
For an arbitrary representation $\rho$ of $Q_{2^n}$ we can write it as $$\rho= m_0.\mb 1  \oplus m_{\eta}\eta\oplus m_{\epsilon} \epsilon   \oplus m_{\eta \epsilon} \eta\epsilon \oplus \bigoplus\limits_{ k~ even} m_k \rho_k \oplus \bigoplus\limits_{k ~odd} m_k \rho_k$$
Taking the total Chern class we get 
$$c(\rho)=(1+b)^{m_{\eta}}(1+a)^{m_{\epsilon}}(1+a+b)^{m_{\epsilon\eta}}\prod\limits_{k ~ odd}(1+k^2e)^{m_k}\prod\limits_{k ~ even}(1+a+k^2e)^{m_k}$$

\section{Realizable Integral Cohomology Classes of $Q_{2^n}$}\label{sec: realizability of Q2n}

A question that arises very naturally is whether every cohomology class in $H^*(Q_{2^n}, \mb Z)$ arises as a Chern class of some representation of $Q_{2^n}$. We recall a result that will be applied in the proof of the next lemma with $p=2.$

\begin{theorem}[\cite{kummer1852}]
The $p$-adic valuation of the binomial coefficient $\binom{n}{t}$ is equal to the number of carry overs when performing the addition of $t$ and $n-t$, written in base $p$.  
\end{theorem}

\begin{lemma}\label{lem: mod 2^n}
    For a fixed $k$, the binomial coefficients $\left\{\binom{m}{k}\mid{m \in \mb Z}\right\}$ take all the residues modulo $2^n.$
\end{lemma}
\begin{proof}
Let $A_k := k - 2^N + \sum_{i=0}^{n-1} r_i 2^{N +i}$, where $2^N$ is the highest power of $2$ less than or equal to $k$ and $r_i \in \{0,1\}$. We claim that 
\begin{equation}\label{eq: Ak mod 2^n}
    \binom{A_k}{k} \equiv r_0 + 2a_{k,1}r_1 + 4a_{k,2}r_2 + \dots + 2^{n-1}a_{k,n-1}r_{n-1} \pmod{2^n}
\end{equation}
where each coefficient $a_{k,i}$ is an odd integer dependent on $k$ and $r_0, \dots, r_{i-1}$. For each $m \in \{1, \dots, n-1\}$, denote    $$ S_m := k - 2^N + \sum_{i=0}^{m} r_i 2^{N+i}$$
    and let $S_0 := k - 2^N + r_0 2^N$. Note that $S_{n-1} = A_k$.
    We first apply Vandermonde's identity to isolate the contribution of the last term $r_{n-1}2^{N+n-1}$:
\begin{equation}\label{eq:vander_general_1}
\binom{A_k}{k} - \binom{A_k - r_{n-1}2^{N+n-1}}{k} = \sum\limits_{j=1}^k \binom{r_{n-1}2^{N+n-1}}{j}\binom{A_k - r_{n-1}2^{N+n-1}}{k-j}
\end{equation}
Observe that $\nu_2\left(\binom{r_{n-1}2^{N+n-1}}{j}\right) \geq n$ for all $j \neq 2^N$. Consequently, every term in the sum vanishes modulo $2^n$ except for the dominant term at $j = 2^N$:
\begin{equation}\label{eq:vander_general_2}
\binom{A_k}{k} - \binom{A_k - r_{n-1}2^{N+n-1}}{k} \equiv \binom{r_{n-1}2^{N+n-1}}{2^N} \binom{A_k - r_{n-1}2^{N+n-1}}{k-2^N} \pmod{2^n}
\end{equation}
Because $2^N \leq k < 2^{N+1}$, the binary expansion of $k-2^N$ has a $0$ at the $2^N$-th position, and its highest nonzero bit appears strictly to the right of the $(N+1)$-st position in the binary expansion. Since $A_k - r_{n-1}2^{N+n-1}$ is constructed by adding multiples of $2^N$ or higher to $k-2^N$, it preserves the lower-order bit structure of $k-2^N$. Thus, by Lucas's theorem we have
\[
\binom{A_k - r_{n-1}2^{N+n-1}}{k-2^N} \equiv 1 \pmod 2
\]
This implies the coefficient is an odd integer, say $a_{k, n-1} \in \mathbb{Z}$. Also, $\binom{r_{n-1}2^{N+n-1}}{2^N} \equiv 2^{n-1}r_{n-1} \pmod{2^n}$. Substituting these back into \eqref{eq:vander_general_2} yields:
\begin{equation}\label{eq:step_n_minus_1}
\binom{A_k}{k} - \binom{A_k - r_{n-1}2^{N+n-1}}{k} \equiv 2^{n-1}a_{k,n-1}r_{n-1} \pmod{2^n} 
\end{equation}
We repeat this process inductively. For $1 \leq m \leq n-2$, Vandermonde's identity yields:
\begin{equation}\label{eq:vander_induction_step}
\binom{S_m}{k} - \binom{S_{m-1}}{k} = \sum\limits_{j=1}^k \binom{r_m 2^{N+m}}{j}\binom{S_{m-1}}{k-j}
\end{equation}
Analyzing modulo $2^{m+1}$, the valuation satisfies $\nu_2\left(\binom{r_m 2^{N+m}}{j}\right) \geq m+1$ for all $j \neq 2^N$. Lucas's theorem ensures that $\binom{S_{m-1}}{k-2^N}$ is odd. Evaluating the surviving term at $j=2^N$ gives:
\begin{equation}\label{eq:step_m}
\binom{S_m}{k} - \binom{S_{m-1}}{k} \equiv 2^m a_{k,m} r_m \pmod{2^{m+1}} \quad \text{where } a_{k,m} \equiv 1 \pmod 2
\end{equation}
where $a_{k,m}$ is an odd integer dependent on $k$ and the lower choices $r_0, \dots, r_{m-1}$.

Descending all the way down to the final two steps, we obtain the following
\begin{equation}\label{eq:step_1}
\binom{k-2^N+r_02^N+r_12^{N+1}}{k} -\binom{k-2^N+r_0 2^N}{k} \equiv 2a_{k,1}r_1 \pmod 4  
\end{equation}
where $a_{k,1} \equiv 1 \pmod 2$, and
\begin{equation}\label{eq:step_0}
 \binom{k-2^N+r_02^N}{k} -\binom{k-2^N}{k} = r_0
\end{equation}

 Telescoping the differences by combining \eqref{eq:step_n_minus_1}, \eqref{eq:step_m} for all intermediate stages, \eqref{eq:step_1} and \eqref{eq:step_0}, recovers \Cref{eq: Ak mod 2^n}.

The map $\Phi_k(r_0,r_1, \dots, r_{n-1}) := \binom{A_k}{k} \pmod{2^n}$ defines a function from  $\{0,1\}^n$ to $\mathbb{Z} / 2^n \mathbb{Z}$. Since both sets contain exactly $2^n$ elements, proving that $\Phi_k$ is injective is enough to conclude that $\Phi_k$ is a bijection.

Suppose there exist $(r_0, \dots, r_{n-1})$ and $(r_0', \dots, r_{n-1}')$ such that $\Phi_k(r_0, \dots, r_{n-1}) \equiv \Phi_k(r_0', \dots, r_{n-1}') \pmod{2^n}$.  
$$
r_0 + \sum_{i=1}^{n-1} 2^i a_{k,i}r_i \equiv r_0' + \sum_{i=1}^{n-1} 2^i a_{k,i}'r_i' \pmod{2^n}
$$
We show that $r_j = r_j'$ for all $j$ using strong induction. For the base case $j=0$, reducing the total congruence modulo $2$ eliminates all terms scaled by $2^i$. Thus
  $$
  r_0 \equiv r_0' \pmod 2 \implies r_0 = r_0' \quad (\text{since } r_0, r_0' \in \{0,1\})
  $$
 Assume that $r_m = r_m'$ holds true for all indices $m < i$. Because the coefficients $a_{k,i}$ are completely determined by $k$ and the lower-order bits $(r_0, \dots, r_{i-1})$, our inductive hypothesis ensures that $a_{k,i} = a_{k,i}'$. Subtracting the identical lower order summation blocks $\sum_{m=0}^{i-1} 2^m a_{k,m} r_m$ from both sides of the main equation obtains:
$$
\sum_{m=i}^{n-1} 2^m a_{k,m}r_m \equiv \sum_{m=i}^{n-1} 2^m a_{k,m}'r_m' \pmod{2^n}
$$
Evaluating the remaining expression modulo $2^{i+1}$ yields:
$$
2^i a_{k,i} r_i \equiv 2^i a_{k,i}' r_i' \pmod{2^{i+1}}
$$
Since $a_{k,i} = a_{k,i}'$, we can divide the entire congruence relation by $2^i$:
$$
a_{k,i} r_i \equiv a_{k,i} r_i' \pmod 2
$$
Because we proved earlier via Lucas's theorem that $a_{k,i}$ is always odd, we have
$$
r_i \equiv r_i' \pmod 2 \implies r_i = r_i'
$$
By induction, this evaluation holds for all $j \in \{0, 1, \dots, n-1\}$. Thus, the two bit-strings must be completely identical, confirming that $\Phi_k$ is an injective mapping and establishing our bijection.
\end{proof}
Combining the preceding discussion with \Cref{lem: mod 2^n}, we arrive at our main result.

\begin{proof}[Proof of \Cref{thm: every cohomology class of Q_2^n is realizable}]
The cohomology classes whose degrees are $0 \text{~mod~} 4$, are of the form $\alpha e^k$, for some integer $\alpha$, and those  whose degrees are $2 \text{~mod~} 4$, are of the form $ be^{k}$ or $ ae^{k}$ or $ (a+b)e^{k}$. Observe that $$c_{2k} \left(\rho_1^{\oplus m}\right)= \binom{m}{k} e^{k} \quad \text{~and~}\quad  c_{2k+1}\left(\rho_1^{\oplus m-1} \oplus\epsilon\right) = \binom{m}{k} be^{k}; $$
$$ c_{2k+1}\left(\rho_1^{\oplus m-1} \oplus\eta\right) = \binom{m}{k} ae^{k} \quad \text{~and~} \quad  c_{2k+1}\left(\rho_1^{\oplus m-1} \oplus\epsilon\eta\right) = \binom{m}{k} (a+b)e^{k}.$$
For the elements of $H^{4k+2}(Q_{2^n}, \mb Z) = \mb Z_2 \times \mb Z_2$ to be realizable, the binomial coefficient $\binom{m}{k}$ has to be odd and taking $m=k$ does the job. For $H^{4k}(Q_{2^n}, \mb Z) = \mb Z_{2^{n}}$, \Cref{lem: mod 2^n} implies that all classes are realizable.
\end{proof}
\section{Realizable Cohomology Classes with $\Z_2$ Coefficients}\label{sec: realizable SW classes}
For an orthogonal complex representation $\pi$ of $G$, there exists a unique real repersentation $\pi_0$, upto isomorphism, so that $\pi\cong \pi_0\otimes_{\mathbb{R}}\mathbb{C}$. We define the Stiefel-Whitney class for $\pi$ as $w_i(\pi)=w_i(\pi_0)$.
%Consider the map $\kappa: H^*(G,\Z)\to H^*(G, \Z_2)$ given by $\kappa(c_i(\pi))=w_{2i}(S(\pi))$.
Recall that
$$\mathcal{H}(G,\Z_2)=\bigcup_{k\geq 1}H^k(G,\Z_2),$$
and
$$\mathcal{RW}(G)=\{ w_i(\pi)\mid \pi\,\, \text{is an orthogonal representation of $G$},\,\,i\geq 1\}.$$

\vspace{1ex}
\subsection{The Realizable Classes $\mathcal{RW}(Q_{2^n})$}~\vspace{1ex}

\begin{theorem}
We have $\mathcal{RW}(Q_{2^n})=\mathcal{H}(Q_{2^n},\Z_2)$.
\end{theorem}

\begin{proof}
The cohomology ring of $Q_{2^n}$ with respect to the ring $\Z_2$ is given by 
$$H^*(Q_{2^n},\Z_2)=\dfrac{\Z_2[X,Y,E]}{\langle XY,X^3+Y^3\rangle},$$
where $X=w_1(\eta), Y=w_1(\epsilon\eta)$ and $E=w_4(S(\rho_1))$.

We compute
$$w_2(S(\eta))=w_2(\eta\oplus \eta^{-1})=w_1(\eta)^2=X^2.$$
Similarly, one obtains 
$$w_2(S(\epsilon\eta))=w_1(\epsilon\eta)^2=Y^2.$$

$$H^{4k}(Q_{2^n}, \Z_2)=\langle E^k \rangle$$

$$w_{4k}(S(\rho_1)^{\oplus k})=E^k.$$

We have $H^{4k+2}(Q_{2^n},\Z_2)=\langle X^2E^k, Y^2E^k \rangle$.

$$w_{4k+2}\left(S(\eta)\oplus S(\rho_1)^{\oplus k})\right)=X^2E^k.$$
Similarly, we have 
$$w_{4k+2}\left(S(\epsilon\eta)\oplus S(\rho_1)^{\oplus k})\right)=Y^2E^k.$$

On the other hand, we have 
$$H^{4k+1}(Q_{2^n},\Z_2)=\{0, XE^k, YE^k, (X+Y)E^k\}.$$
We have $w_{4k+1}(\eta\oplus S(\rho_1)^{\oplus k})=XE^k$, $w_{4k+1}(\epsilon\eta\oplus S(\rho_1)^{\oplus k})=YE^k$, and $w_{4k+1}(\epsilon\oplus S(\rho_1)^{\oplus k})=(X+Y)E^k.$ Also, $H^{4k+1}(Q_{2^n},\Z_2)=\langle X^3E^k \rangle $. One computes that $w_{4k+3}(\eta^{\oplus 3}\oplus S(\rho_1)^{\oplus k})=X^3E^k$.
\end{proof}

\vspace{1ex}
\subsection{The Realizable Classes $\mathcal{RW}(D_{2n})$}
~\vspace{1ex}

Consider the representation 
\begin{equation}
\pi=m_1\mathbb{1}\oplus m_r\chi_r\oplus m_s\chi_s\oplus m_{rs}\chi_{rs}\oplus m_{\sigma}\sigma_1.    
\end{equation}
Then we have 
$$w(\pi)=(1+y)^{m_r}(1+x)^{m_s}(1+x+y)^{m_{rs}}(1+x+w)^{m_{\sigma}}.$$
We have $(1+y)(1+x+y)=(1+x+2y+y(x+y))=(1+x)$.
This gives
\begin{align*}
w(\pi) &= (1+y)^{m_r+m_s}(1+x+y)^{m_{rs}+m_s}(1+x+w)^{m_{\sigma}}\\
&=(1+y)^{d_1}(1+x+y)^{d_2}(1+x+w)^{d_3}\\
&=(1+y)^{d_1}(1+x')^{d_2}(1+x+w)^{d_3}
\end{align*}
where $d_1=m_r+m_s$, $d_2=m_{rs}+m_s$ and $d_3=m_{\sigma}$.

\begin{theorem}
We have $\mathcal{RW}(D_{2n})\neq \mathcal{H}(D_{2n}, \mathbb{Z}_2)$.    
\end{theorem}

\begin{proof}
    
We show that $x^3+xw\notin \mathcal{RW}(D_{2n})$.
We have
\begin{align*}
&[w(\pi)]_{3^{-}}=  \left(1+d_1y+\binom{d_1}{2}y^2+\binom{d_1}{3}y^3\right)\left( 1+d_2x'+\binom{d_2}{2}x'^2+\binom{d_2}{3}x'^3\right)\\
&\left(1+d_3(x+w)+\binom{d_3}{2}(x+w)^2+\binom{d_3}{3}(x+w)^3\right)  \\
=& \left(1+d_1y+\binom{d_1}{2}y^2+\binom{d_1}{3}y^3+d_2(x+y)+\binom{d_2}{2}(x^2+y^2)+\binom{d_2}{3}(x^3+y^3) \right)\\
&\left(1+d_3(x+w)+\binom{d_3}{2}(x^2+w^2)+\binom{d_3}{3}x^3\right)\\
=& A_1x^3+A_2x^2y+A_3xy^2+A_4y^3+A_5xw.
\end{align*}
The coefficients $A_i$ are given by
\begin{align*}
A_1 &= \binom{d_2}{3}+d_2\binom{d_3}{2}+\binom{d_3}{3}+d_3\binom{d_2}{2}\\
A_2&= d_1\binom{d_3}{2}+d_2\binom{d_3}{2}\\
A_3&=d_3\binom{d_1}{2}+d_3\binom{d_2}{2}\\
A_4&=\binom{d_1}{3}+\binom{d_2}{3}\\
A_5&=d_2d_3.
\end{align*}
If $A_5$ is odd, then both $d_2$ and $d_3$ are odd. But in that case $A_1=\binom{d_2+d_3}{3}$ is even. Therefore $x^3+xw\notin \mathcal{RW}(D_{2n})$.

Note that if $m\equiv 1\pmod 4$, then both $\binom{m}{2}$ and $\binom{m}{3}$ are even. So, if we choose $d_i\equiv 1 \pmod 4$ for $1\leq i\leq 3$, then the terms $A_2, A_3, A_4$ will be even. 
\end{proof}

 \bibliographystyle{alpha} 
\bibliography{mybib}

@book{adem2013cohomology,
  title={Cohomology of finite groups},
  author={Adem, Alejandro and Milgram, R James},
  volume={309},
  year={2013},
  publisher={Springer Science \& Business Media}
}

@article{handel1993products,
  title={On products in the cohomology of the dihedral groups},
  author={Handel, David},
  journal={Tohoku Mathematical Journal, Second Series},
  volume={45},
  number={1},
  pages={13--42},
  year={1993},
  publisher={Mathematical Institute, Tohoku University}
}

@book{hirzebruch1966topological,
  title={Topological methods in algebraic geometry},
  author={Hirzebruch, Friedrich and Borel, Armand and Schwarzenberger, RLE},
  volume={175},
  year={1966},
  publisher={Springer Berlin-Heidelberg-New York}
}

@article{hayami2002cohomology,
  title={Cohomology ring of the generalized quaternion group with coefficients in an order},
  author={Hayami, Takao and Sanada, Katsunori},
  journal={Communications in Algebra},
  volume={30},
  number={8},
  pages={3611--3628},
  year={2002},
  publisher={Taylor \& Francis}
}

@article{pitt1973free,
  title={Free actions of generalized quaternion groups on spheres},
  author={Pitt, David},
  journal={Proceedings of the London Mathematical Society},
  volume={3},
  number={1},
  pages={1--18},
  year={1973},
  publisher={Oxford University Press}
}

@article {kummer1852,
    AUTHOR = {Kummer, E. E.},
     TITLE = {\"Uber die {E}rg\"anzungss\"atze zu den allgemeinen
              {R}eciprocit\"atsgesetzen},
   JOURNAL = {J. Reine Angew. Math.},
  FJOURNAL = {Journal f\"ur die Reine und Angewandte Mathematik. [Crelle's
              Journal]},
    VOLUME = {44},
      YEAR = {1852},
     PAGES = {93--146},
      ISSN = {0075-4102,1435-5345},
   MRCLASS = {99-04},
  MRNUMBER = {1578793},
       DOI = {10.1515/crll.1852.44.93},
       URL = {https://doi.org/10.1515/crll.1852.44.93},
}

@article{atiyah1961characters,
  title={Characters and cohomology of finite groups},
  author={Atiyah, Michael F},
  journal={Publications Math{\'e}matiques de l'IH{\'E}S},
  volume={9},
  pages={23--64},
  year={1961}
}

@article{joshi2026characteristic,
  title={Characteristic Classes Of Representations Of Lie Groups},
  author={Joshi, Rohit and Spallone, Steven},
  journal={arXiv preprint arXiv:2602.02145},
  year={2026}
}

@article{naolekar2012realizing,
  title={Realizing cohomology classes as Euler classes},
  author={Naolekar, Aniruddha C},
  journal={Mathematica Slovaca},
  volume={62},
  number={5},
  pages={949--966},
  year={2012},
  publisher={Springer}
}

\end{document}